\documentclass{amsart}

\usepackage[mathscr]{euscript}
\usepackage{amssymb}
\usepackage{amsmath}
\usepackage{latexsym}
\usepackage{color}

\makeatletter
\renewcommand\@biblabel[1]{#1.}
\makeatother
\newtheorem{thm}{Theorem}
\newtheorem{lem}{Lemma}
\newtheorem{prop}{Proposition}

\newtheorem{cor}{Corollary}
\theoremstyle{definition}
\newtheorem{problem}{Problem}

\newtheorem{df}{Definition}
\newtheorem{remark}{Remark}
\theoremstyle{definition}
\newtheorem{ex}{Example}

\newcommand\mf{\mathfrak }

\newcommand{\Fin}{\textrm{Fin}} 
\DeclareMathOperator\supp{supp}
\DeclareMathOperator\card{card}
\DeclareMathOperator\Exh{Exh}
\DeclareMathOperator\iden{id}
\newcommand{\id}{\iden_\omega}
\newcommand{\ce}{\mf c}
\newcommand{\eps}{\varepsilon}
\newcommand{\I}{\mathcal I}
\newcommand{\J}{\mathcal J}
\newcommand{\Z}{\mathcal Z}
\newcommand{\EU}{\mathcal{EU}}

\title[Erd\H{o}s-Ulam ideals vs. simple density ideals]{Erd\H{o}s-Ulam ideals vs. simple density ideals}

\author{Adam Kwela}
\thanks{The author has been supported by the grant BW-538-5100-B482-17.}
\address{Institute of Mathematics, Faculty of Mathematics, Physics and Informatics, University of Gda\'{n}sk, ul.~Wita Stwosza 57, 80-308 Gda\'{n}sk, Poland}
\email{adam.kwela@ug.edu.pl}

\begin{document}
\begin{abstract}
The main aim of this paper is to bridge two directions of research generalizing asymptotic density zero sets. This enables to transfer results concerning one direction to the other one.

Consider a function $g\colon\omega\to [0,\infty)$ such that $\lim_{n\to\infty}g(n)=\infty$ and $\frac{n}{g(n)}$ does not converge to $0$. Then the family $\Z_g=\{A\subseteq\omega:\ \lim_{n\to\infty}\frac{\card(A\cap n)}{g(n)}=0\}$ is an ideal called simple density ideal (or ideal associated to upper density of weight $g$). We compare this class of ideals with Erd\H{o}s-Ulam ideals. In particular, we show that there are $\sqsubseteq$-antichains of size $\ce$ among Erd\H{o}s-Ulam ideals which are and are not simple density ideals (in \cite{generalized} it is shown that there is also such an antichain among simple density ideals which are not Erd\H{o}s-Ulam ideals).

We characterize simple density ideals which are Erd\H{o}s-Ulam as those containing the classical ideal of sets of asymptotic density zero. We also characterize Erd\H{o}s-Ulam ideals which are simple density ideals. In the latter case we need to introduce two new notions. One of them, called increasing-invariance of an ideal $\I$, asserts that given $B\in\I$ and $C\subseteq\omega$ with $\card(C\cap n)\leq\card(B\cap n)$ for all $n$, we have $C\in\I$. This notion is inspired by \cite{Inv} and is later applied in \cite{generalized} for a partial solution of \cite[Problem 5.8]{zJackiem}.

Finally, we pose some open problems.
\end{abstract}
\maketitle

\section{Introduction}

We denote the set $\{0,1,\ldots\}$ by $\omega$ and identify $n$ with $\{0,\ldots,n-1\}$. A family $\mathcal{I}$ of subsets of $\omega$ is called an \emph{ideal on $\omega$} if it is closed under taking finite unions and subsets. We additionally assume that each ideal is a proper subset of $\mathscr{P}(\omega)$ and contains all finite subsets of $\omega$.

One of the most classical examples of ideals is 
$$\Z=\left\{A\subseteq\omega:\ \lim_{n\to\infty}\frac{\card(A\cap n)}{n}=0\right\}$$ 
-- the \emph{ideal of asymptotic density zero sets} (cf. \cite[Example 1.2.3.(d)]{Farah}). This ideal has been deeply investigated in the past in the context of convergence (see e.g. \cite{Fast}, \cite{Fridy}, \cite{Salat} and \cite{Steinhaus}) as well as from the set-theoretic point of view (see e.g. \cite{Farah} and \cite{Just}).

W. Just and A. Krawczyk in \cite{Just} (where they solved a question raised by P. Erd\H{o}s and S. Ulam) introduced a generalization of the above. If $f\colon\omega\to [0,\infty)$ is such that $\sum_{i=0}^\infty f(i)=\infty$ and $\lim_{n\to\infty}\frac{f(n)}{\sum_{i=0}^n f(i)}=0$, then we define an \emph{Erd\H{o}s-Ulam ideal $\EU_f$} by 
$$A\in\EU_f\Leftrightarrow\lim_{n\to\infty}\frac{\sum_{i\in A\cap n} f(i)}{\sum_{i=0}^{n-1} f(i)}=0.$$

Over many years other ways of generalizing the ideal $\Z$, asymptotic density or the notion of convergence associated to $\Z$ have been considered (cf. \cite{1}, \cite{Den}, \cite{BoseDas}, \cite{3}). This paper is an attempt to bridge those directions of research in order to transfer results between them. We will concentrate on two ways of generalizing the ideal of asymptotic density zero sets.

Recently, in \cite{Den} the authors proposed another generalization of the ideal $\Z$. Consider a function $g\colon\omega\to [0,\infty)$ such that $\lim_{n\to\infty}g(n)=\infty$ and $\frac{n}{g(n)}$ does not converge to $0$. Then the family $\Z_g=\{A\subseteq\omega:\ \lim_{n\to\infty}\frac{\card(A\cap n)}{g(n)}=0\}$ is an ideal called \emph{simple density ideal} (note that this name was introduced in \cite{generalized} -- before that, although those ideals have been extensively studied, they did not have their own name and were called \emph{ideals associated to upper density of weight $g$} or similarly). Note that the condition that $\frac{n}{g(n)}$ does not converge to $0$ ensures us that $\omega\notin\Z_g$. It is easy to see that $\Z_g=\Z_{\lfloor g\rfloor}$ for each such function $g$. What is more, by \cite[Proposition 2.2]{Den}, for any such $g$ one can find a nondecreasing $f$ with $\Z_f = \Z_g$. Therefore, we can restrict our considerations only to ideals $\Z_g$, where 
$$g\in H=\left\{f\colon\omega\to\omega:\ f\textrm{ is nondecreasing}\ \wedge\ f(n)\rightarrow\infty\ \wedge\ \frac{n}{f(n)}\nrightarrow 0\right\}.$$

Papers \cite{Den} and \cite{generalized} are devoted to investigations of simple density ideals. For instance, in \cite{Den} it is proved that all simple density ideals are $\mathbf{F}_{\sigma\delta}$ but not $\mathbf{F}_{\sigma}$ P-ideals (for definition of a P-ideal see \cite[Section 1.2.]{Farah}).

One motivation for studying simple density ideals is related to various notions of convergence -- the class of simple density ideals (as well as density functions related to them) has been intensively studied in this context. For an ideal on $\omega$ we say that a sequence of reals $(x_n)$ is \emph{$\I$-convergent to $x\in\mathbb{R}$} if $\{n\in\omega:\ |x_n-x|\geq\eps\}\in\I$ for each $\eps>0$. For instance, in \cite{5} the authors studied which subsequences of a given $\Z_g$-convergent sequence, where $g\in H$, are also $\Z_g$-convergent. In \cite{6} the notion of $\Z_g$-convergence has been compared (in the context of fuzzy numbers) with the notion of so-called \emph{$\I$-lacunary statistical convergence of weight $g$}. Paper \cite{4} is in turn an investigation of a variant of convergence associated to simple density ideals and matrix summability methods. 

Another motivation comes from the following question (cf. \cite[Problem 5.8]{zJackiem}): for which ideals $\I$, if $\I\upharpoonright A$ is isomorphic to $\I$ for some $A\subseteq\omega$, then the witnessing isomorphism $\phi\colon\omega\to A$ is just an increasing enumeration of $A$ (for definition of isomorphic ideals see end of this Section). In \cite{generalized} it was shown that all simple density ideals have this property (in case of $\Z$ it was known earlier, cf. \cite[Theorem 5.6]{zJackiem}). However, this is not true for all Erd\H{o}s-Ulam ideals (it is easy to see that if $f(n)=1$ for odd $n$'s and $f(n)=0$ for even $n$'s, then $\EU_f\upharpoonright\omega\setminus\{0\}$ is isomorphic to $\EU_f$, but the increasing enumeration of $\omega\setminus\{0\}$ is not an isomorphism). Moreover, there are not many known examples of ideals with the above property. Therefore, another result from \cite{generalized}, stating that there are $\ce$ many isomorphic types of simple density ideals, becomes especially significant. 

It is easy to see that the classes of Erd\H{o}s-Ulam ideals and simple density ideals have nonempty intersection (for instance $\Z$ is in both those classes). However, there is no inclusion between them (cf. \cite[Examples 3.1 and 3.2]{Den}). The aim of this paper is to investigate further relations between those classes. In particular, it would be interesting to know which Erd\H{o}s-Ulam ideals have the property from the previous paragraph. 

In this context measures supported by $\omega$ will occur to be a very convenient tool. For such measure $\mu$ by $\supp(\mu)$ we denote its \emph{support}, i.e., the set $\{n\in\omega:\ \mu(\{n\})\neq 0\}$. We say that an ideal is a \emph{density ideal} (in sense of Farah) if it is of the form
$$\Exh\left(\sup_{n\in\omega}\mu_n\right)=\left\{A\subseteq\omega:\ \lim_{n\to\infty}\mu_n(A)=0\right\},$$
where $(\mu_n)$ is a sequence of measures on $\omega$ with pairwise disjoint and finite supports (cf. \cite[Section 1.13.]{Farah}). Note that this is a special case of a more general notion -- in this case $\varphi=\sup_{n\in\omega}\mu_n$ is a lower semicontinuous submeasure supported by $\omega$ and $\Exh(\varphi)$ is the exhaustive ideal generated by $\varphi$. S. Solecki proved in \cite{Solecki} that every analytic P-ideal is equal to the exhaustive ideal generated by some lower semicontinuous submeasure $\varphi$.

The importance of density ideals in our studies is a consequence of \cite[Theorem 3.2]{Den} and \cite[Theorem 1.13.3.(a)]{Farah} -- Erd\H{o}s-Ulam ideals as well as simple density ideals are density ideals (in the case of Erd\H{o}s-Ulam ideals we can even assume that all $\mu_n$ are probability measures with supports being consecutive intervals). Moreover, thanks to I. Farah we have the following characterization.

\begin{thm}(\cite[Theorem 1.13.3.(b)]{Farah})
\label{FarahEU}
A density ideal $\Exh(\sup_{n\in\omega}\mu_n)$ is an Erd\H{o}s-Ulam ideal if and only if the following conditions hold:
\begin{enumerate}
\item[(D1)] $\sup_{k\in\omega} \mu_k(\omega)<\infty$,
\item[(D2)] $\lim_{k\to\infty} \sup_{i\in\omega} \mu_k(\{i\})=0$,
\item[(D3)] $\limsup_{k\to\infty} \mu_k(\omega)>0$.
\end{enumerate}
\end{thm}

In our further considerations we will also need orders on ideals. Let $\I$ and $\J$ be two ideals on $\omega$. We say that:
\begin{itemize}
 \item \emph{$\I$ and $\J$ are isomorphic} ($\mathcal{I}\cong\mathcal{J}$)  if there is a bijection $\phi\colon\omega\to\omega$ such that $A\in\I\Longleftrightarrow\phi^{-1}[A]\in\J$ for all $A\subseteq\omega$;
  \item \emph{$\I$ is below $\J$ in the Rudin-Blass order} ($\mathcal{I}\leq_{RB}\mathcal{J}$) if there is a finite-to-one function $\phi\colon\omega\to\omega$ such that $A\in\I\Longleftrightarrow\phi^{-1}[A]\in\J$ for all $A\subseteq\omega$;
 \item \emph{$\J$ contains an isomorphic copy of $\I$} ($\mathcal{I}\sqsubseteq\mathcal{J}$) if there is a bijection $\phi\colon\omega\to\omega$ such that $A\in\I\Longrightarrow\phi^{-1}[A]\in\J$ for all $A\subseteq\omega$;
 \item \emph{$\I$ is below $\J$ in the Kat\v{e}tov order} ($\mathcal{I}\leq_{K}\mathcal{J}$) if there is a function (not necessarily a bijection) $\phi\colon\omega\to\omega$ such that $A\in\I\Longrightarrow\phi^{-1}[A]\in\J$ for all $A\subseteq\omega$. 
\end{itemize}
It is obvious that $\mathcal{I}\cong\mathcal{J}$ is the strongest among the above notions and $\mathcal{I}\leq_K\mathcal{J}$ is the weakest. We will also use the following: if $\preceq$ is some order on ideals, then we say that:
\begin{itemize}
 \item $\I$ and $\J$ are \emph{$\preceq$-equivalent} if $\mathcal{I}\preceq\mathcal{J}$ and $\mathcal{J}\preceq \mathcal{I}$;
 \item a family of ideals $\mathcal{F}$ is an \emph{$\preceq$-antichain} if $\mathcal{I}\not\preceq\mathcal{J}$ and $\mathcal{J}\not\preceq\mathcal{I}$ for every pair of ideals $\I,\J\in\mathcal{F}$.
\end{itemize}

A property of ideals can often be expressed by finding a critical ideal (in sense of some order on ideals) with respect to this property (see \cite{WR} and \cite{Solecki}). This approach is very effective for instance in the context of ideal convergence of sequences of functions (see \cite{zReclawem} and \cite{zMarcinem}). One such result regarding simple density ideals can be found in Theorem \ref{simple-EU}.

The paper is organized as follows. In Section 2 we apply orders on ideals on $\omega$ to studies of simple density ideals. Section 3 is devoted to characterizing simple density ideals which simultaneously are Erd\H{o}s-Ulam. We give a characterization which does not use the function generating the simple density ideal. In Section 4 we introduce the notion of increasing-invariant ideals and then use it in Section 5 to construct antichains of Erd\H{o}s-Ulam ideals which are and which are not simple density ideals. Section 6 contains our main result -- a characterization of Erd\H{o}s-Ulam ideals which simultaneously are simple density ideals. This part of the paper is divided into three subsections in which we prove some lemmas, introduce the notion of almost uniformly distributed ideals and prove the main result. Finally, in Section 7 we pose some open problems.

\section{When a simple density ideal is an Erd\H{o}s-Ulam ideal?}

The aim of this section is to characterize simple density ideals which simultaneously are Erd\H{o}s-Ulam ideals. One such characterization was established in \cite{generalized}: if $g\in H$, then the ideal $\Z_g$ is an Erd\H{o}s-Ulam ideal if and only if the sequence $\left(\frac{\card(g^{-1}[[2^n,2^{n+1})])}{2^n}\right)_{n\in\omega}$ is bounded. We give several other (less technical) equivalent conditions. In particular, we describe simple density Erd\H{o}s-Ulam ideals in terms of orders on ideals. This approach enables us to express the fact that a simple density ideal $\Z_g$ is Erd\H{o}s-Ulam without using the function $g$.

We start with a result which turns out to be convenient for our further considerations.

\begin{prop}
\label{EU<=>bounded}
The simple density ideal $\Z_g$ is an Erd\H{o}s-Ulam ideal if and only if the sequence $(n/g(n))_{n\in\omega}$ is bounded.
\end{prop}

\begin{proof}
By \cite[Theorem 3.2]{Den} we know that $\Z_g=\Exh(\sup_{k\in\omega}\mu_k)$, where $\mu_k(A)=\frac{\card(A\cap [n_k,n_{k+1}))}{g(n_k)}$ and $n_0=1$, $n_{k+1}=\min\{n\in\omega:\ g(n)\geq 2g(n_k)\}$ for $k\in\omega$. Then, by \cite[Theorem 1.13.3.(b)]{Farah} the ideal $\Z_g$ is an Erd\H{o}s-Ulam ideal if and only if the following conditions hold:
\begin{enumerate}
\item[(D1)] $\sup_{k\in\omega} \mu_k(\omega) < \infty$,
\item[(D2)] $\lim_{k\to\infty} \sup_{i\in\omega} \mu_k(\{i\}) = 0$,
\item[(D3)] $\limsup_{k\to\infty} \mu_k(\omega)>0$.
\end{enumerate}
Condition (D3) is satisfied as it is equivalent to $\omega\notin\Z_g$ (cf. \cite[discussion above Theorem 3.3]{Den}). Condition (D2) is satisfied since all simple density ideals are tall (cf. \cite[discussion above Proposition 1.1]{Den}) and (D2) is equivalent to $\Exh(\sup_{k\in\omega}\mu_k)$ being tall (\cite[Theorem 3.3]{Den}). Therefore, it remains to check that (D1) is equivalent to $(n/g(n))$ being bounded. 

Suppose that there is $\delta>0$ such that $n/g(n)<\delta$ for all $n$. Then 
$$\mu_k(\omega)=\frac{n_{k+1}-n_k}{g(n_k)}\leq\frac{n_{k+1}}{g(n_k)}<\frac{2n_{k+1}}{g(n_{k+1}-1)}=\frac{2(n_{k+1}-1)}{g(n_{k+1}-1)}+\frac{2}{g(n_{k+1}-1)}<\delta+\frac{2}{g(0)}$$
for all $k$.

Suppose now that there is $\delta>0$ such that $\mu_k(\omega)<\delta$ for all $k$. Then for each $k$ and $n_k<i\leq n_{k+1}$ we have
$$\frac{i}{g(i)}\leq\frac{n_{k+1}}{g(n_k)}=\frac{n_0}{g(n_k)}+\sum_{i=0}^{k}\frac{n_{i+1}-n_i}{g(n_k)}\leq\frac{n_0}{g(0)}+\sum_{i=0}^{k}\frac{n_{i+1}-n_i}{2^{k-i}g(n_i)}\leq \frac{n_0}{g(0)}+2\delta.$$
\end{proof}

Our next result shows that $\Z$ is a critical ideal for Erd\H{o}s-Ulam ideals among all simple density ideals. This characterization does not use the function generating a simple density ideal. 

\begin{thm}
\label{simple-EU}
If $\I$ is a simple density ideal, then the following conditions are equivalent:
\begin{enumerate}
\item $\I$ is an Erd\H{o}s-Ulam ideal;
\item $\Z\subseteq\I$;
\item $\Z\leq_K\I$.
\end{enumerate}
\end{thm}

\begin{proof}
The implication (2)$\implies$(3) is trivial. Therefore, we only need to prove (1)$\implies$(2) and (3)$\implies$(1). Let $f\in H$ be such that $\Z_f=\I$. We will use Proposition \ref{EU<=>bounded}.

(1)$\implies$(2): Suppose that there is $\delta>0$ such that $n/f(n)<\delta$ for all $n\in\omega$. We will show that $\Z_g\subseteq \Z_f$ for $g\colon\omega\to\mathbb{R}_+$ given by $g(n)=n/\delta$. Take any $B\in\Z_g$. Then 
$$\lim_{n\to\infty}\frac{\card(B\cap n)}{f(n)}\leq\lim_{n\to\infty}\delta\frac{\card(B\cap n)}{n}=\lim_{n\to\infty}\frac{\card(B\cap n)}{g(n)}=0.$$
Now it remains to observe that $\Z_g=\Z$. Indeed, we have $\delta\lim_{n\to\infty}\card(A\cap n)/n=\lim_{n\to\infty}\card(A\cap n)/g(n)$ for every $A\subseteq\omega$.

(3)$\implies$(1): Suppose that $\Z_f$ is not an Erd\H{o}s-Ulam ideal and fix any function $\phi\colon\omega\to\omega$. We will show that there is $A\in\Z$ such that $\phi^{-1}[A]\notin\Z_f$. There is a sequence $(m_n)$ such that $m_n/f(m_n)>2n+3$ for all $n\in\omega$. We may additionally assume that $(2n+2)f(m_n)<f(m_{n+1})$. Note that this implies $\sum_{i=0}^{n-1}f(m_i)\leq nf(m_{n-1})<f(m_{n})$. 

Consider the intervals $I_n=(f(m_n),(2n+2)f(m_n)]$, for $n\in\omega$, and denote:
$$B_{n}=\{i\in I_{n}:\ \phi(i)>\max I_n\};$$
$$C_{n}=\{i\in I_{n}:\ \phi(i)\in I_n\};$$
$$D_{n}=\{i\in I_{n}:\ \phi(i)<\min I_n\}.$$
At least one of the following three cases must happen. 

\textbf{Case 1.: }For infinitely many $n\in\omega$ we have $\card(B_{n})\geq f(m_n)$. Let $(n_j)_{j\in\omega}$ be an increasing enumeration of the set of $n$'s with such property. For each $j\in\omega$ let $A_{j}$ be any subset of $B_{n_j}$ of cardinality $f(m_n)$. Define $A=\bigcup_{j\in\omega}A_{j}$. Then we have $\frac{\card(\phi[A]\cap i)}{i}\leq \frac{2f(m_n)}{(2n+2)f(m_n)}\leq\frac{2}{2n+2}$ for all $\max I_n<i\leq\max I_{n+1}$. On the other hand, $\card(\phi^{-1}[\phi[A]]\cap m_n)/f(m_n)\geq 1$. Hence, $\phi[A]\in\Z$ and $\phi^{-1}[\phi[A]]\notin\Z_f$.

\textbf{Case 2.: }For infinitely many $n\in\omega$ we have $\card(C_{n})\geq nf(m_n)$. Let $(n_j)_{j\in\omega}$ be an increasing enumeration of the set of $n\neq 0$ with such property. For each $j\in\omega$ and $l=0,1,\ldots,f(m_{n_j})-1$ pick $a_{j,l}\in(f(m_{n_j})+l(2n_j+1),f(m_{n_j})+(l+1)(2n_j+1)]$ such that $\card(\phi^{-1}[\{a_{j,l}\}]\cap I_{n_j})$ is as big as possible among all points from $(f(m_{n_j})+l(2n_j+1),f(m_{n_j})+(l+1)(2n_j+1)]$. 

Define $A_j=\{a_{j,l}:\ l=0,1,\ldots,f(m_{n_j})-1\}$ and $A=\bigcup_{j\in\omega}A_j$. Note that $\card(\phi^{-1}[A_j]\cap I_{n_j})\geq n_j f(m_{n_j})/(2n_j+1)$. Indeed, otherwise we would have $\card(C_{n_j})\leq (2n_j+1)\card(\phi^{-1}[A_j]\cap I_{n_j})<n_j f(m_{n_j})$, a contradiction. Therefore, we have $\card(\phi^{-1}[A]\cap m_{n_j})/f(m_{n_j})\geq n_j/(2n_j+1)\geq 1/3$ (recall that $n_j>0$). It follows that $\phi^{-1}[A]\notin\Z_{f}$. However, for all $\min I_{n_j}\leq i<\min I_{n_{j+1}}$ we have
$$\frac{\card(A\cap i)}{i}\leq\frac{2f(m_{n_{j-1}})+(l+1)}{f(m_{n_j})+l(2n_j+1)}<\frac{2f(m_{n_{j-1}})}{f(m_{n_j})}+\frac{l+1}{l(2n_j+1)}\leq\frac{2}{2n_j+2}+\frac{2}{2n_j+1}$$ 
if $i\in(f(m_{n_j})+l(2n_j+1),f(m_{n_j})+(l+1)(2n_j+1)]$ and
$$\frac{\card(A\cap i)}{i}\leq\frac{2f(m_{n_j})}{f(m_{n_j})(2n_j+2)}<\frac{1}{n_j+1}$$
if $i>\max I_{n_j}$. Hence, $A\in\Z$.

\textbf{Case 3.: }For infinitely many $n\in\omega$ we have $\card(D_{n})\geq n f(m_n)$. Let $N$ consist of all $n$'s with such property. We inductively pick an increasing sequence $(A_j)$ of finite sets. At step $j$ let $t_{j}>\max A_{j-1}$ be such that $\card(A_{j-1})/t_{j}<1/j$. Find such $n_j\in N$ that $\card(D_{n_j}\setminus\phi^{-1}[[0,t_j)])\geq n_j f(m_{n_j})/2$ (if such $n_j$ does not exist, then we are done as $\phi^{-1}[[0,t_j)]\notin\Z_{f}$ by $\card(\phi^{-1}[[0,t_j)]\cap m_{n})/f(m_n)\geq n/2\geq 1/2$ for all $n\in N$). We can additionally assume that $1/(f(m_{n_j})-t_j)<1/j$. For each $l=0,1,\ldots,\lceil (f(m_{n_j})-l_j)/n_j\rceil-1$ pick 
$$a_{j,l}\in\left(t_j+l\frac{f(m_{n_j})-t_j}{\lceil (f(m_{n_j})-t_j)/n_j\rceil},t_j+(l+1)\frac{f(m_{n_j})-t_j}{\lceil (f(m_{n_j})-t_j)/n_j\rceil}\right]$$
such that $\card(\phi^{-1}[\{a_{j,l}\}]\cap I_{n_j})$ is as big as possible. Let $A_j=\{a_{j,l}:\ l=0,1,\ldots,\lceil (f(m_{n_j})-l_j)/n_j\rceil-1\}$.

Define $A=\bigcup_{j\in\omega}A_j$. Note that $\card(\phi^{-1}[A_j]\cap I_{n_j})\geq f(m_{n_j})/2$. Indeed, otherwise we would have 
$$\card(D_{n_j}\setminus\phi^{-1}[[0,t_j)])<\frac{f(m_{n_j})-t_j}{\lceil (f(m_{n_j})-t_j)/n_j\rceil}\frac{f(m_{n_j})}{2}\leq\frac{n_j f(m_{n_j})}{2},$$ 
a contradiction. Therefore, we have $\card(\phi^{-1}[A]\cap m_{n_j})/f(m_{n_j})\geq 1/2$. It follows that $\phi^{-1}[A]\notin\Z_{f}$. However, for all $t_j\leq i\leq f(m_{n_j})$ we have
$$\frac{\card(A\cap i)}{i}\leq\frac{\card(A_{j-1})+l+1}{t_j+l\frac{f(m_{n_j})-t_j}{\lceil (f(m_{n_j})-t_j)/n_j\rceil}}<$$
$$<\frac{\card(A_{j-1})}{t_j}+\frac{l+1}{l\frac{f(m_{n_j})-t_j}{\lceil (f(m_{n_j})-t_j)/n_j\rceil}}\leq\frac{1}{j}+2\frac{1+(f(m_{n_j})-t_j)/n_j}{f(m_{n_j})-t_j}\leq\frac{3}{j}+\frac{2}{n_j},$$ 
where 
$$i\in\left(t_j+l\frac{f(m_{n_j})-t_j}{\lceil (f(m_{n_j})-t_j)/n_j\rceil},t_j+(l+1)\frac{f(m_{n_j})-t_j}{\lceil (f(m_{n_j})-t_j)/n_j\rceil}\right]$$
and $\card(A\cap i)/i=\card(A\cap f(m_{n_j}))/f(m_{n_j})$ for all $f(m_{n_j})<i<t_{j+1}$. Hence, $A\in\Z$.
\end{proof}

In \cite[Theorem 1.13.10]{Farah} I. Farah proved that all Erd\H{o}s-Ulam ideals are $\leq_{RB}$-equivalent. Theorem \ref{simple-EU} enables us to improve this result in the case of simple density ideals -- we obtain equivalence instead of implication in only one direction.

\begin{cor}
\label{cor}
A simple density ideal is an Erd\H{o}s-Ulam ideal if and only if it is $\leq_{RB}$-equivalent to $\Z$.
\end{cor}

\begin{proof}
By \cite[Theorem 1.13.10]{Farah} all Erd\H{o}s-Ulam ideals are Rudin-Blass equivalent. This gives us the implication from left to right. On the other hand, if some $\Z_f$ is $\leq_{RB}$-equivalent to $\Z$ then $\Z\leq_K\Z_f$. Hence, $\Z_f$ is an Erd\H{o}s-Ulam ideal by the equivalence of (1) and (3) from Theorem \ref{simple-EU}.
\end{proof}

\begin{remark}
\label{EU:RB<->K}
Note that by \cite[Theorem 4]{generalized} we have $\I\leq_K\Z$ for each simple density ideal $\I$. Hence, $\leq_{RB}$-equivalence in Corollary \ref{cor} can be replaced by $\leq_{K}$-equivalence. Equivalence of (2) from previous Theorem \ref{simple-EU} and the condition from Proposition \ref{EU<=>bounded} should be in turn compared with \cite[Lemma 5]{generalized} stating that $\Z_f\subseteq\Z$ if and only if $\liminf_{n\to\infty}n/f(n)>0$. 
\end{remark}

\section{Increasing-invariance}

In this section we introduce the notion of increasing-invariant ideals. It will be very useful in characterizing Erd\H{o}s-Ulam ideals which simultaneously are simple density ideals.

\begin{df}
We say that an ideal $\I$ is \emph{increasing-invariant} if for every $B\in\I$ and $C\subseteq\omega$ satisfying $\card(C\cap n)\leq\card(B\cap n)$ for all $n$, we have $C\in\I$.
\end{df}

\begin{remark}
The above notion is inspired by \cite[Section 4]{Inv}, where the authors studied ideals $\I$ satisfying the following: for any increasing injection $f\colon\omega\to\omega$ condition $B\in\I$ implies $f[B]\in\I$. It is easy to see that the above is equivalent to increasing-invariance of $\I$.
\end{remark}

\begin{remark}
Sometimes an ideal $\I$ is called shift-invariant if $A\in\I$ implies $\{a+k:\ a\in A\}\cap\omega\in\I$ for every $k\in\mathbb{Z}$ (note that here the shift is the same for all points from $A$ while our notion of increasing-invariance allows different shifts for different points).
\end{remark}

Let us point out that increasing-invariance is equivalent to the following (perhaps more natural) statement.

\begin{prop}
\label{delta}
An ideal $\I$ is increasing-invariant if and only if for every $B\in\I$ and $C\subseteq\omega$ such that there is $\delta>0$ satisfying $\card(C\cap n)\leq\delta\card(B\cap n)$ for all $n$, we have $C\in\I$.
\end{prop}

\begin{proof}
The implication from right to left is obvious. In order to prove the converse implication, assume that $\I$ is increasing-invariant and fix $B\in\I$ and $C\subseteq\omega$ such that there is $\delta>0$ with $\card(C\cap n)\leq\delta\card(B\cap n)$ for all $n$. If $\delta\leq 1$, then $\card(C\cap n)\leq\delta\card(B\cap n)\leq\card(B\cap n)$ for all $n$. Hence, $C\in\I$ by the increasing-invariance of $\I$. If $\delta>1$, then let $(c_i)$ be an increasing enumeration of the set $C$. Define $C_k=\{c_{k+i\lceil\delta\rceil}:\ i\in\omega\setminus\{0\}\}$ for $k=0,1,\ldots,\lceil\delta\rceil-1$. Observe that 
$$\card(C_k\cap n)\leq\frac{\card(C\cap n)}{\lceil\delta\rceil}\leq\frac{\delta\card(B\cap n)}{\lceil\delta\rceil}\leq\card(B\cap n)$$
for each $k$. Hence, $C_k\in\I$ for each $k$ and 
$$C=\{c_0,\dots,c_{\lceil\delta\rceil-1}\}\cup C_0\cup\ldots\cup C_{\lceil\delta\rceil-1}\in\I.$$
\end{proof}

Next fact establishes connection of increasing-invariance with the class of simple density ideals. Note that it was also observed in \cite{Inv} and \cite{generalized}. Actually, in Section \ref{sekcja} we show that this connection is much deeper. However, the following result will be sufficient to construct in Section \ref{antichains} an antichain of Erd\H{o}s-Ulam ideals which are not simple density ideals. 

\begin{prop}(\cite[Section 4]{Inv} and \cite[Proposition 11]{generalized})
\label{wprawo}
Every simple density ideal is increasing-invariant.
\end{prop}

\begin{proof}
Fix $g\in H$ and suppose that $B\in\Z_g$. Then for any $C\subseteq\omega$ satisfying $\card(C\cap n)\leq\card(B\cap n)$ for all $n$, we have $\card(C\cap n)/g(n)\leq\card(B\cap n)/g(n)$ for every $n$. Hence, $C\in\Z_g$.
\end{proof}

We end this section with some remarks on increasing-invariance in case of ideals which are not simple density ideals.

Recall that an ideal is \emph{tall} if any infinite set in $\mathscr{P}(\omega)$ contains an infinite subset belonging to the ideal. Denote by $\Fin$ the ideal consisting of all finite subsets of $\omega$.

\begin{prop}
$\Fin$ is a increasing-invariant ideal. Moreover, it is the only increasing-invariant ideal among all non-tall ideals. 
\end{prop}

\begin{proof}
The first part is trivial. To prove the second part fix a non-tall ideal $\I\neq\Fin$, an infinite set $B\in\I$ and $A\notin\I$ such that every its infinite subset is not in $\I$. Let $(b_i)$ be an increasing enumeration of $B$. Define $c_i=\min\{a\in A:\ a\geq b_i\}$ and $C=\{c_i:\ i\in\omega\}$. Then $\card(C\cap n)\leq\card(B\cap n)$ for all $n$. Moreover, $C\subseteq A$ is infinite, so $C\notin\I$.
\end{proof}

\begin{remark}
\label{rem}
Observe that there are tall and increasing-invariant ideals which are not simple density ideals (or even density ideals). A good example is the ideal $\I_{1/n}=\{A\subseteq\omega:\ \sum_{a\in A}1/a<\infty\}$ -- it is not a simple density ideal (it is $\mathbf{F}_{\sigma}$ by \cite[Example 1.2.3.(c)]{Farah} and simple density ideals are not $\mathbf{F}_{\sigma}$ by \cite[Corollary 3.5.]{Den}).
\end{remark}

\section{Antichains}
\label{antichains}

In \cite{generalized} it is proved that among simple density ideals there is an $\leq_{K}$-antichain of size $\ce$. Since all Erd\H{o}s-Ulam ideals are Rudin-Blass equivalent (by \cite[Theorem 1.13.10]{Farah}), we can infer that actually it is an $\leq_{K}$-antichain among simple density ideals which are not Erd\H{o}s-Ulam ideals. In this section we construct antichains among Erd\H{o}s-Ulam ideals. In this case best possible is to have an antichain in the sense of $\sqsubseteq$.

\begin{prop}
There is an $\sqsubseteq$-antichain of size $\ce$ among Erd\H{o}s-Ulam ideals which are not simple density ideals.
\end{prop}

\begin{proof}
Let $n_0=1$ and $n_{i+1}=n_i+i!+(i+1)!$ for all $i\in\omega$. Define $\mu_i(\{j\})=1/i!$ for $j\in D_i=\{n_i,n_i+1,\ldots,n_i+i!-1\}$ and $\mu_i(\{j\})=0$ otherwise. Then each $\mu_i$ is a probability measure with finite support. Fix a family $\mathcal{F}$ of cardinality $\ce$ of infinite pairwise almost disjoint subsets of $\omega$. For each $M\in\mathcal{F}$ let $\I_M=\Exh(\sup_{m\in M}\mu_m)$. Then $\I_M$ is an Erd\H{o}s-Ulam ideal. Moreover, by Proposition \ref{wprawo} each $\I_M$ is not a simple density ideal, since it is not increasing-invariant (consider $C=\bigcup_{m\in M\setminus\{0\}}D_m\notin\I_M$ and $B=\bigcup_{m\in M\setminus\{0\}}[n_m-m!,n_m-1)\in\I_M$).

Fix $M,K\in\mathcal{F}$. Now we will show that $\I_M\not\sqsubseteq\I_K$. Let $\phi\colon\omega\to\omega$ be any bijection and $(k_i)_{i\in\omega}$ be an increasing enumeration of $K\setminus (M\cup\{0,1,\ldots,5\})$. Observe that for $n>5$ we have $n!/2>2\sum_{i=1}^{n-1} i!$. Therefore, for each $i\in\omega$ we can find a set $A_i\subseteq\phi(D_{k_i})$ of cardinality $k_i!/2$, such that $\min A_i\geq m_{k_i}-k_i!$. Notice that $A=\bigcup_{i\in\omega}A_i\in\I_M$ since for each $m\in M$ with $m\geq 5$ we have 
$$\mu_m(A)\leq\frac{\frac{1}{2}(6!+7!+\ldots+(m-1)!)}{m!}<\frac{(m-1)!}{m!}=\frac{1}{m}.$$
However, $\phi^{-1}[A]\notin\I_K$ by $\phi_{k_i}(A)=1/2$ for all $i$.
\end{proof}

\begin{prop}
There is an $\sqsubseteq$-antichain of size $\ce$ among Erd\H{o}s-Ulam simple density ideals.
\end{prop}

\begin{proof}
Let $(n_i)$ be any sequence satisfying $n_{i+1}>in_i$. Fix a family $\mathcal{F}$ of cardinality $\ce$ of infinite pairwise almost disjoint subsets of $\omega$. For each $L\in\mathcal{F}$ let $(l_i)_{i\in\omega}$ be its increasing enumeration and define
$$f_L(n)=\left\{\begin{array}{ll}
l_i n_{l_i} & \textrm{if } n_{l_i}<n\leq l_i n_{l_i}\textrm{ for some }i\in\omega,\\
n & \textrm{otherwise.}
\end{array}\right.$$
It is easy to see that $\limsup_{n\to\infty}n/f_L(n)=1$ and $n/f_L(n)\leq 1$ for all $n$. Hence, $f_L\in H$ and $\Z_{f_L}$ is an Erd\H{o}s-Ulam ideal (cf. Proposition \ref{EU<=>bounded}).

Fix $K,M\in\mathcal{F}$. We will show that $\Z_{f_M}\not\sqsubseteq\Z_{f_K}$. Suppose that $\phi\colon\omega\to\omega$ is any bijection. Let $(m_i)_{i\in\omega}$ be an increasing enumeration of $M\setminus (K\cup\{0,1,2\})$. For each $i$ at least $n_{m_i}$ elements of $\phi[[n_{m_i},3n_{m_i})]$ are above $n_{m_i}-1$, so let $A_i\subseteq\phi[[n_{m_i},3n_{m_i})]\setminus n_{m_i}$ be such that $\card(A_i)=n_{m_i}$. Define $A=\bigcup_{i\in\omega}A_i$. Then we have $\card(\phi^{-1}[A]\cap 3n_{m_i})/f_K(3n_{m_i})\geq 1/3$ and 
$$\frac{\card(A\cap n)}{f_M(n)}\leq \frac{2n_{m_i}}{m_i n_{m_i}}=\frac{2}{m_i}$$
for all $n_{m_i}<n\leq n_{m_{i+1}}$. Therefore, $\phi^{-1}[A]\notin\Z_{f_K}$ and $A\in\Z_{f_M}$.
\end{proof}

\section{When an Erd\H{o}s-Ulam ideal is a simple density ideal?}
\label{sekcja}

\subsection{Some lemmas}

Simple density ideals have less complicated definition than Erd\H{o}s-Ulam ideals. Hence, it would be interesting to know in case of which Erd\H{o}s-Ulam ideals we can use this simpler form. The aim of this section is to characterize Erd\H{o}s-Ulam ideals which are simple density ideals. This is much more complicated than characterizing simple density ideals which are Erd\H{o}s-Ulam. Therefore, we will need some lemmas.

By Proposition \ref{wprawo} and Theorem \ref{simple-EU} we know that necessary conditions for an Erd\H{o}s-Ulam ideal $\I$ to be simple density ideal are that $\I$ is increasing-invariant and $\Z\subseteq\I$. Next result shows that first condition implies second. However, we do not know any proof of our characterization (cf. Theorem \ref{main}) which does not use the fact that $\Z\subseteq\I$.

\begin{lem}
\label{zawieraId}
If $\I$ is a increasing-invariant Erd\H{o}s-Ulam ideal, then $\Z\subseteq\I$.
\end{lem}

\begin{proof}
By \cite[Theorem 1.13.3.(a)]{Farah} we have a sequence $(\mu_i)$ of probability measures such that $\I=\Exh(\sup_{i\in\omega}\mu_i)$. Actually, by the proof of \cite[Theorem 1.13.3.(a)]{Farah} we can even assume that $D_i=\supp(\mu_i)$ are consecutive intervals. Fix any $C\in\Z$ and suppose to the contrary that there is some $\delta>0$ such that $\mu_i(C)>\delta$ for infinitely many $i$. We will inductively construct sets $B_k$ and $C_k$ for $k\in\omega\setminus\{0\}$. Suppose that we have already defined $B_l$ and $C_l$ for $l<k$. There is $n_k$ such that:
\begin{itemize}
\item[(i)] $2^{n_k}>\max C_{k-1}$ if $k>1$;
\item[(ii)] for any $n>2^{n_k}$ and $i\in\omega$ we have $\mu_i(\{n\})<1/k$;
\item[(iii)] $\card(C\cap (2^n,2^{n+1}])/2^n<1/(2k)$ for all $n>n_k$.
\end{itemize} 
Find $i_k$ such that $\min D_{i_k}>2^{n_k+1}$ and $\mu_{i_k}(C)>\delta$. Put $C_k=C\cap D_{i_k}$ and denote $l_k=\min\{j\in\omega:\ C_k\cap (2^j,2^{j+1}]\neq\emptyset\}$ and $r_k=\max\{j\in\omega:\ C_k\cap (2^j,2^{j+1}]\neq\emptyset\}$. For each $l_k\leq j\leq r_k$ pick $\card(C_k\cap (2^{j},2^{j+1}])$ elements of $I_j=(2^{j-1},2^{j}]$ in such a way that:
\begin{itemize}
\item[(a)] at most $\card(E_j)/k$ elements of $E_j=\bigcup_{i\in T_j}D_i\cap I_j$ were picked, where $T_j=\{i\in\omega:\ D_i\cap I_j\neq\emptyset\ \wedge\ D_i\not\subseteq I_j\}$;
\item[(b)] if for some $i$ we have $D_i\subseteq I_j$, then at most $l+1$ elements of $D_i$ were picked, where $lk<\card(D_i)\leq (l+1)k$ (note that $\card(D_i)\geq k$ by condition (ii));
\item[(c)] the value of $\sup_{i\in\omega}\mu_i$ on each picked element is the least possible.
\end{itemize} 
Let $B_k$ consist of the picked elements. This ends the construction.

It is easy to see that $\card(C_k\cap n)\leq\card(B_k\cap n)$ for all $n$ and $k$. Moreover, $\mu_i(B_k)\leq 2/k$ for all $i$ and $k$. Indeed, take $l_k\leq j\leq r_k$ and note that if $D_i\subseteq (2^{j-1},2^{j}]$, then $\mu_i(B_k)\leq(l+1)/\card(D_i)\leq (l+1)/(lk)\leq 2/k$, where $l$ is as in (b). What is more, $\sum_{i\in T_j}\mu_i(I_j)\leq 2$ (recall that $\mu_i$ are probability measures and $D_i$ are consecutive intervals), so $\mu_i(B_k)\leq\frac{2}{\card(E)}\frac{\card(E)}{k}=2/k$, where $T_j$ and $E_j$ are as in (a). 

Finally, observe that $B=\bigcup_{k\in\omega\setminus\{0\}}B_k\in\I$, $C'=\bigcup_{k\in\omega\setminus\{0\}}C_k\notin\I$ and $\card(C'\cap n)\leq\card(B\cap n)$ for all $n$, which contradicts the fact that $\I$ is increasing-invariant.
\end{proof}

The next example shows that there is an Erd\H{o}s-Ulam ideal $\I$ such that $\Z\subseteq\I$ but $\I$ is not increasing-invariant.

\begin{ex}
Let $D_k=(2k!,3k!]$ for all $k\in\omega\setminus\{0\}$. Define a sequence of probability measures $(\mu_k)$ by $\mu_k(\{i\})=1/(k!)$ for all $i\in D_k$ and $\mu_k(\{i\})=0$ for all $i\notin D_k$. Let $\I$ be the Erd\H{o}s-Ulam ideal associated with $(\mu_k)$. Observe first that $\I$ is not increasing-invariant since for $C=\bigcup_{k\in\omega\setminus\{0\}}D_k\notin\I$ and $B=\bigcup_{k\in\omega\setminus\{0\}}((k!),2(k!)]\in\I$ we have $\card(C\cap n)\leq\card(B\cap n)$ for every $n$. Now we will show that $\Z\subseteq\I$. Fix any $A\notin\I$. There are $\delta>0$ and a sequence $(m_k)$ such that $\mu_{m_k}(A)>\delta$ for all $k$. It follows that $\card(A\cap D_k)>\delta (k!)$. Then we have 
$$\frac{\card(A\cap 3(k!))}{3(k!)}\geq\frac{\card(A\cap D_k)}{3(k!)}>\frac{\delta}{3},$$
hence, $A\notin\Z$.
\end{ex}

Next two lemmas show that increasing-invariance enables us to have a better control of measures generating the Erd\H{o}s-Ulam ideal.

\begin{lem}
\label{lem2}
If $\I$ is an increasing-invariant density ideal generated by a sequence of measures $(\mu_n)$, then it is equal to $\J=\{A\subseteq\omega:\ \phi^{-1}[A]\in\I\}$, where $\phi\colon\omega\to\omega$ is a bijection given by $\phi[\supp(\mu_n)]=\supp(\mu_n)$ and
$$\phi_n(i)\leq\phi_n(j)\Leftrightarrow \mu_n(\{i\})\geq\mu_n(\{j\})$$
for all $n\in\omega$ and $i,j\in\supp(\mu_n)$.
\end{lem}

\begin{proof}
We will show that $\J\subseteq\I$. The proof of $\I\subseteq\J$ is similar. 

Fix any $A\in\J$. Then $\mu_n(\phi^{-1}[A])$ tends to $0$. Observe that $A'=\bigcup_{n\in\omega}\{a\in A\cap\supp(\mu_n):\ \mu(\{a\})\leq\mu_n(\{\phi^{-1}(a)\})\}\in\I$. 

For each $a\in A\cap\supp(\mu_n)$, since $a=\phi(\phi^{-1}(a))$, there are at least $\card(D_n\setminus a)$ elements $b\in\supp(\mu_n)$ such that $\mu_n(\{b\})\leq\mu_n(\{\phi^{-1}(a)\})$. Therefore, if $a\notin A'$, then we can find $b_a\in\supp(\mu_n)$ such that $\mu_n(\{b_a\})\leq\mu_n(\{\phi^{-1}(a)\})$ and additionally $b_a\leq a$. 

Consider now the set $B=\{b_a:\ a\in A\setminus A'\}$. It is in $\I$ since $\phi^{-1}[A\setminus A']\in\I$ and $\mu_n(\{b_a\})\leq\mu_n(\{\phi^{-1}(a)\})$. By $b_a\leq a$ and the increasing-invariance of $\I$ we get that $A\setminus A'\in\I$. Hence, $A\in\I$.
\end{proof}

\begin{lem}
\label{lem1}
If $\I$ is an increasing-invariant Erd\H{o}s-Ulam ideal, then it is generated by a sequence $(\mu_i)$ of probability measures with $\card(\supp(\mu_{i+1}))\geq 2\card(\supp(\mu_i))$ and $\max\supp(\mu_i)<\min\supp(\mu_{i+1})$ for all $i\in\omega$. Moreover, we can assume that $\bigcup_{i\in\omega}\supp(\mu_i)=\omega$.
\end{lem}

\begin{proof}
Let $(\mu'_n)$ be the sequence of probability measures generating $\I$. By the proof of \cite[Theorem 1.13.3(a)]{Farah} we can assume that $\bigcup_{i\in\omega}\supp(\mu'_i)=\omega$ and $\supp(\mu'_i)$ are consecutive intervals.

Denote $D_i=\supp(\mu'_i)$ for $i\in\omega$. Define inductively natural numbers $i_n$ and sets $E_n$ for all $n\in\omega$. Start with $i_0=0$ and $E_0=D_0$. Suppose now that $i_m$ and $E_m$ are defined for all $m\leq n$ and let $i_{n+1}=\min\{k>i_n:\ \card(D_{i_n+1}\cup\ldots\cup D_{k})\geq 2\card(E_n)\}$ and $E_{n+1}=D_{i_n+1}\cup\ldots\cup D_{i_{n+1}}$.

Denote $i_{-1}=-1$. For each $m\in\omega$ define $\mu_m(\{k\})=\mu'_i(\{k\})/(i_m-i_{m-1})$ if $k\in D_i$ for some $i_{m-1}<i\leq i_m$ and $\mu_m(\{k\})=0$ otherwise. Then $\supp(\mu_m)=E_m$ and $\mu_n$ is a probability measure. 

Now we will show that there is $k_0$ such that for any $m$ we have $i_m-i_{m-1}\leq k_0$. Suppose otherwise: there is a sequence $(m_k)$ such that $i_{m_k}-i_{m_k-1}>k$ for all $k$. Hence, for each $k>0$ there is $i_{m_k-1}<j_k\leq i_{m_k}$ such that $\card(D_{j_k})<2\card(E_{m_k-1})/k$. We have
$$\frac{\card(D_{j_k})}{\max E_{m_k-1}}<\frac{2\card(E_{m_k-1})/k}{\max E_{m_k-1}}\leq\frac{2}{k},$$
but $\bigcup_{k\in\omega}D_{j_k}\notin\I$. A contradiction with Lemma \ref{zawieraId}.

We need to show that $\I=\Exh(\sup_{n\in\omega}\mu_n)$. Fix any $A\subseteq\omega$. Obviously, if $A\in\I$, then $A\in\Exh(\sup_{n\in\omega}\mu_n)$. On the other hand, if $A\notin\I$, then $\mu'_n(A)$ is greater than some $\delta>0$ for infinitely many $n$. Hence, $\mu_n(A)>\delta/k_0$ for infinitely many $n$.
\end{proof}

The next lemma is rather technical -- it expresses an observation which we will use several times in further considerations.

\begin{lem}
\label{lem3}
Assume that $\Z_g$ is a simple density Erd\H{o}s-Ulam ideal and $(\mu_n)$ is the sequence of measures from Lemma \ref{lem1}. If $B\notin\Z_g$, then there are $\delta>0$ and a sequence $(l_j)$ such that $\frac{\card(B\cap \supp(\mu_{n_j})\cap l_j)}{g(l_j)}\geq\delta$, where $l_j\in \supp(\mu_{n_j})$. 
\end{lem}

\begin{proof}
Denote $D_n=\supp(\mu_n)$. Observe first that there is some $\alpha>0$ such that $\card(D_{n})/g(i)<\alpha$ for all $n\in\omega$ and $i\in D_{n+1}$. Indeed, otherwise we would have a sequence $(t_n)$ such that $\card(D_{m_n})/g(t_n)>n$ for all $n$, where $t_n\in D_{m_n+1}$. Then let $A_n=[t_n-\lfloor \card(D_{m_n})/n\rfloor,t_n)$ and define $A=\bigcup_{n\in\omega}A_n$. We have $A\in\Z$ as $\card(A_n)/\max A_n\leq\card(D_{m_n})/(n\max D_{m_n})<1/n$. However, $\card(A\cap t_n)/g(t_n)\geq \card(D_{m_n})/(ng(t_n))-1/g(t_n)>1-1/g(t_n)$ which is greater than $1/2$ if $t_n$ is sufficiently large. A contradiction with Lemma \ref{zawieraId}.

Assume now that for some $B\subseteq\omega$ we have $B\notin\Z_g$, i.e., there are $m\in\omega$ and a sequence $(i_j)$ such that $\card(B\cap i_j)/g(i_j)>\alpha/2^m$ for all $j\in\omega$. Define $\delta=\frac{\alpha}{2^{m+2} (m+3)}$. Fix $j$ and denote by $s_j$ the unique $s\in\omega$ with $i_j\in D_s$. Then either $\card(B\cap D_{s_j}\cap i_j)/g(i_j)\geq 2\delta$ or there must be $s_j-(m+2)\leq n_j<s_j$ such that $\card(B\cap D_{n_j})/g(\max D_{n_j})\geq 2\delta$. Indeed, otherwise we would have
$$\frac{\alpha}{2^m}<\frac{\card(B\cap i_j)}{g(i_j)}\leq$$
$$\frac{\card(B\cap D_{s_j}\cap i_j)}{g(i_j)}+\frac{\card(B\cap D_{s_j-1})}{g(\max D_{s_j-1})}+\ldots+\frac{\card(B\cap D_{s_j-(m+2)})}{g(\max D_{s_j-(m+2)})}+\frac{\max D_{s_j-(m+3)}}{g(i_j)}$$
which is less than $\alpha/2^m$ since $\max D_{s_j-(m+3)}/g(i_j)<\alpha/2^{m+1}$ as
$$\max D_{s_j-(m+3)}<\card(D_{s_j-(m+2)})<\frac{\card(D_{s_j-1})}{2^{m+1}}\leq\frac{\alpha g(i_j)}{2^{m+1}}.$$
Therefore, $\frac{1+\card(B\cap D_{n_j}\cap l_j)}{g(l_j)}\geq 2\delta$ for some $s_j-(m+2)\leq n_j\leq s_j$ and $l_j\in D_{n_j}$. As $\lim_{i\to\infty}g(i)=\infty$, if $l_j$ is sufficiently large, then we get $\frac{\card(B\cap D_{n_j}\cap l_j)}{g(l_j)}\geq\delta$.
\end{proof}

\subsection{Almost uniform distribution}

It occurs that the assumption that an Erd\H{o}s-Ulam ideal is increasing-invariant is not enough for it to be a simple density ideal. In this subsection we prove it and introduce another, besides increasing-invariance, necessary condition -- almost uniform distribution.

\begin{df}
\label{star}
Let $\Exh(\sup_{n\in\omega}\mu_n)$ be a density ideal. We say that a set $B\subseteq\omega$ \emph{satisfies condition $(\star)$ with respect to $\Exh(\sup_{n\in\omega}\mu_n)$} if $\{i\in B:\ \mu_n(\{i\})<\frac{1}{d_n}\}=\emptyset$ and
$$\forall_{\genfrac{}{}{0pt}{1}{m\in\omega}{m>0}}\ \exists_{n_m\in\omega}\ \forall_{n>n_m}\ \forall_{\genfrac{}{}{0pt}{1}{k\in\omega}{k>0}}\ \card\left(\left\{i\in B:\ \frac{k}{d_n}\leq\mu_n(\{i\})<\frac{k+1}{d_n}\right\}\right)\leq\frac{d_n}{mk(k+1)},$$
where $d_n$ denotes $\card(\supp(\mu_n))$. We say that $\Exh(\sup_{n\in\omega}\mu_n)$ is \emph{almost uniformly distributed} if every $B$ satisfying condition $(\star)$ with respect to $\Exh(\sup_{n\in\omega}\mu_n)$ belongs to $\Exh(\sup_{n\in\omega}\mu_n)$.
\end{df}

\begin{remark}
\label{n_m increasing}
Notice that the sequence $(n_m)$ in condition $(\star)$ can be assumed to be increasing.
\end{remark}

\begin{remark}
\label{jednostajnie}
Observe that condition $(\star)$ with respect to $\Exh(\sup_{n\in\omega}\mu_n)$ is equivalent to $\sigma_{B,n}(k)\rightrightarrows 0$ if $n\to\infty$, where 
$$\sigma_{B,n}(k)=k(k+1)\frac{\card\left(\left\{i\in B:\ \frac{k}{d_n}\leq\mu_n(\{i\})<\frac{k+1}{d_n}\right\}\right)}{d_n}$$
(as before, $d_n$ denotes $\card(\supp(\mu_n))$).
\end{remark}

Now we prove that almost uniform distribution is a necessary condition for an Erd\H{o}s-Ulam ideal to be a simple density ideal. 

\begin{lem}
\label{wprawo'}
Every simple density Erd\H{o}s-Ulam ideal is almost uniformly distributed.
\end{lem}

\begin{proof}
Let $\Exh(\sup_{n\in\omega}\mu_n)$ be a simple density Erd\H{o}s-Ulam ideal and denote $D_n=\supp(\mu_n)$. Take $g\in H$ such that $\Z_g=\Exh(\sup_{n\in\omega}\mu_n)$. By Proposition \ref{wprawo} $\Z_g$ is increasing-invariant. Hence, by Lemmas \ref{lem2} and \ref{lem1} we can assume that $(D_n)$ is a partition of $\omega$ into consecutive intervals and $\mu_n(\{i-1\})\geq\mu_n(\{i\})$ for all $n\in\omega$ and $i\in D_n$, $i>\min D_n$.

Suppose to the contrary that $\Z_g$ is not almost uniformly distributed and $B'\notin\Z_g$ is the witnessing subset of $\omega$. Then by Lemma \ref{lem3} there are $\delta>0$ and a sequence $(i_j)$ such that $\card(B'\cap i_j\cap D_{s_j})/g(i_j)>\delta$ for all $j\in\omega$, where by $s_j$ we denote the unique $s$ with $i_j\in D_{s}$. Let $B=B'\cap\bigcup_{j\in\omega}D_{s_j}$. Then $B\notin\Z_g$ and $B$ satisfies condition $(\star)$ with respect to $\Exh(\sup_{n\in\omega}\mu_n)$.

Fix $j\in\omega$ and denote by $m_j$ the unique $m$ with $n_m<s_j\leq n_{m+1}$, where $n_m$ are as in Definition \ref{star} (by Remark \ref{n_m increasing} we may additionally assume that $(n_m)$ is increasing). Consider the set $B_j=B\cap D_{s_j}\cap i_j$. Define $C_j$ as the set consisting of the last $\lceil\card(B_j)/2\rceil$ elements of $B_j$. 

Denote 
$$L^n_k=\left\{i\in D_n:\ \frac{k}{\card(D_n)}\leq\mu_n(\{i\})<\frac{k+1}{\card(D_n)}\right\}$$
for all $n,k\in\omega$ and observe that
$$\card\left(\bigcup_{k\geq\bar{k}}L^{s_j}_k\cap B\right)\leq\sum_{k=\bar{k}}^{\infty}\card\left(L^{s_j}_k\cap B\right)\leq\sum_{k=\bar{k}}^{\infty}\frac{1}{m_j k(k+1)}\leq\frac{1}{m_j \bar{k}}$$
for every $\bar{k}\in\omega\setminus\{0\}$.

Let $l\in\omega$ be minimal such that $\card(B_j)\geq\card(D_{s_j})/l$ (note that $l\geq m_j$ by our observation from the previous paragraph) and $t\in\omega$ be maximal such that $t m_j\leq l-1$. The value of $\mu_{s_j}(C_j)$ will be the biggest possible if $B_j\subseteq \bigcup_{k\geq t}L^{s_j}_k$ (note that $\card(\bigcup_{k\geq t}L^{s_j}_k\cap B)\leq\card(D_{s_j})/(tm_j)$ by our observation and $\card(B_j)\leq\card(D_{s_j})/(l-1)\leq\card(D_{s_j})/(tm_j)$). Then $C_j\subseteq\bigcup_{k=t}^{2(t+1)}L^{s_j}_k$ since 
$$\card\left(\bigcup_{k>2(t+1)}L^{s_j}_k\cap B\right)\leq\frac{\card(D_{s_j})}{m_j2(t+1)}\leq\frac{\card(D_{s_j})}{2l}\leq\frac{\card(B_j)}{2}\leq\card(C_j).$$
We have
$$\mu_{s_j}(C_j)\leq \sum_{k=t}^{2(t+1)}\frac{(k+1)\card(C_j\cap L^{s_j}_k)}{\card(D_{s_j})}\leq\sum_{k=t}^{2(t+1)}\frac{(k+1)\card(B_j\cap L^{s_j}_k)}{\card(D_{s_j})}\leq$$
$$\sum_{k=t}^{2(t+1)}\frac{1}{m_j k}\leq\frac{1}{m_j}\frac{t+3}{t}\leq \frac{4}{m_j}.$$
Hence, $C=\bigcup_{j\in\omega}C_j$ belongs to $\Exh(\sup_{n\in\omega}\mu_n)=\Z_g$.

On the other hand, $\card(C\cap i_j)/g(i_j)\geq \card(B\cap i_j)/(2g(i_j))>\delta/2$. A contradiction.
\end{proof}

Next example shows that increasing-invariance is not strong enough to imply almost uniform distribution.

\begin{ex}
Denote $k_n=\min\{k\in\omega:\ \frac{1}{n}\sum_{i=1}^{k}\frac{1}{i+1}\geq \frac{1}{2}\}$. Let $(D_n)$ be a sequence of consecutive intervals such that:
\begin{itemize}
\item[(i)] $\frac{\card(D_n)}{n k_n (k_n+1)}\geq\card(D_{n-1})$ for all $n>0$;
\item[(ii)] $nk(k+1)|\card(D_n)$ for all $n\in\omega$ and $k=1,2,\ldots,k_n$.
\end{itemize}
Let also $L^n_{k}$ for $n\in\omega$ and $0\leq k\leq k_n$ be such that:
\begin{itemize}
\item $L^n_{k_n}=[\min D_{n},\min D_{n}+\frac{\card(D_n)}{n k_n (k_n+1)}-1]$, i.e., $L^n_{k_n}$ is the beginning section of $D_{n}$ of length $\frac{\card(D_n)}{n k_n (k_n+1)}$;
\item $L^n_{k}=[\min D_{n}\setminus\bigcup_{i=k}^{k_n}L^n_i,\min D_{n}\setminus\bigcup_{i=k}^{k_n}L^n_i+\frac{\card(D_n)}{n k (k+1)}-1]$, i.e., $L^n_{k}$ is the interval of length $\frac{\card(D_n)}{n k (k+1)}$ starting just after the end of $L^n_{k+1}$;
\item $L^n_0=D_n\setminus\bigcup_{k=1}^{k_n}L^n_k$.
\end{itemize}
For each $n$ define a measure $\mu_n\colon\mathscr{P}(\omega)\to [0,1]$ as follows:
\begin{itemize}
\item $\mu_n(\{i\})=\frac{k}{\card(D_n)}$ if $i\in L^n_k$ for some $k>1$;
\item $\mu_n(\{i\})=0$ for $i\notin D_n$;
\item $\mu_n$ on $L^n_0$ is uniformly distributed in such a way that $\mu_n(\omega)=1$.
\end{itemize}
Then $\Exh(\sup_{n\in\omega}\mu_n)$ is not almost uniformly distributed as witnessed by the set $\bigcup_{n\in\omega}\bigcup_{k=1}^{k_n}L^n_k$ since
$$\mu_n\left(\bigcup_{k=1}^{k_n}L^n_k\right)=\sum_{k=1}^{k_n}\frac{\card(D_n)}{nk(k+1)}\cdot\frac{k}{\card(D_n)}=\frac{1}{n}\sum_{k=1}^{k_n}\frac{1}{k+1}\geq \frac{1}{2}.$$
However, it is increasing-invariant by condition (i).
\end{ex}

We end this subsection with a proof of the fact that actually there is no implication between almost uniform distribution and increasing-invariance.

\begin{prop}
\label{bounded}
If the sequence $(\card(\supp(\mu_n))\mu_n(\{i\}))_{n,i\in\omega}$ is bounded, then the ideal $\Exh(\sup_{n\in\omega}\mu_n)$ is almost uniformly distributed.
\end{prop}

\begin{proof}
As before, denote $\card(\supp(\mu_n))$ by $d_n$ and suppose that $B\subseteq\omega$ satisfies condition $(\star)$ with respect to $\Exh(\sup_{n\in\omega}\mu_n)$. Then for every $m\in\omega\setminus\{0\}$ and $n>n_m$ (where $n_m$ are as in Definition \ref{star}) we have $\mu_n(B)\leq\sum_{k=1}^{\infty}\frac{d_n}{mk(k+1)}\frac{k+1}{d_n}=\sum_{k=1}^{\infty}\frac{1}{mk}$. Let $M\in\omega$ be such that $d_n\mu_n(\{i\})\leq M$ for all $n,i\in\omega$ and observe that actually $\mu_n(B)\leq\sum_{k=1}^{M}\frac{1}{mk}$. Therefore, $\lim_{n\to\infty}\mu_n(B)=0$ and $B\in \Exh(\sup_{n\in\omega}\mu_n)$.
\end{proof}

\begin{ex}
Let $(D_n)$ be a sequence of consecutive intervals such that $\card(D_n)=2^{n+1}$ for all $n\in\omega$. Define $\mu_n(\{i\})=1/2^n$ for $\max D_n-2^n<i\leq\max D_n$ and $\mu_n(\{i\})=0$ otherwise. Then $\Exh(\sup_{n\in\omega}\mu_n)$ is not increasing-invariant. However, it is almost uniformly distributed as $\card(\supp(\mu_n))\mu_n(\{i\})\leq 2$ for all $n,i\in\omega$ (cf. Proposition \ref{bounded}). 
\end{ex}

\subsection{Main result}

Finally, we are ready to characterize Erd\H{o}s-Ulam ideals which simultaneously are simple density ideals.

\begin{thm}
\label{main}
An Erd\H{o}s-Ulam ideal is a simple density ideal if and only if it is increasing-invariant and almost uniformly distributed.
\end{thm}

\begin{proof}
The implication from left to right is given by Proposition \ref{wprawo} and Lemma \ref{wprawo'}. We will prove the converse implication.

Let $(\mu_n)$ be the sequence of probability measures generating an Erd\H{o}s-Ulam ideal $\I$ and denote $D_n=\supp(\mu_n)$. By Lemmas \ref{lem2} and \ref{lem1} we can assume that $(D_n)$ is a partition of $\omega$ into consecutive intervals and $\mu_n(\{i-1\})\geq\mu_n(\{i\})$ for all $n\in\omega$ and $i\in D_n$, $i>\min D_n$.

For all $n,k\in\omega$, $k\neq 0$, define:
$$L^n_k=\left\{i\in D_n:\ \frac{k}{\card(D_n)}\leq\mu_n(\{i\})<\frac{k+1}{\card(D_n)}\right\};$$
$$R^n_k=\left\{i\in D_n:\ \frac{1}{(k+1)\card(D_n)}\leq\mu_n(\{i\})<\frac{1}{k\card(D_n)}\right\}.$$
Observe that $D_n=\bigcup_{k\in\omega\setminus\{0\}}(L^n_k\cup R^n_k)$ for each $n$.

Define $h\colon\omega\to\omega$ by
$$h(i)=\left\{\begin{array}{ll}
\frac{\card(D_n)}{k} & \textrm{if }i\in L^n_k\textrm{ for some }n\in\omega\textrm{ and }k\in\omega\setminus\{0\},\\
(k+1)\card(D_n) & \textrm{if }i\in R^n_k\textrm{ for some }n\in\omega\textrm{ and }k\in\omega\setminus\{0\}.
\end{array}\right.$$
Notice that $h$ does not have to be nondecreasing although it is nondecreasing on each $D_n$. 

Put $L_0=\emptyset$ and for each $n\in\omega$ let $R_n\subseteq D_n$ and $L_{n+1}\subseteq D_{n+1}$ be such that:
\begin{itemize}
\item[(a)] $R_n$ is the final part of $D_n$, i.e., is equal to $[j,\max D_n]$ for some $j\in D_n$;
\item[(b)] $L_{n+1}$ is the initial part of $D_{n+1}$, i.e., is equal to $[\min D_{n+1},j]$ for some $j\in D_{n+1}$;
\item[(c)] $\card(R_n)\leq\card(D_n)/2$ and $\mu_{n+1}(L_{n+1}\setminus\{\max L_{n+1}\})<1/2$;
\item[(d)] $R_n\cap L_n=\emptyset$;
\item[(e)] cardinalities of $R_n$ and $L_{n+1}$ are least possible such that:
\begin{itemize}
\item $\card(R_n)=\card(L_{n+1})$ and $h(\min R_n)\leq h(\max L_{n+1})$ if this is consistent with the above conditions;
\item if it is not, then the one of the sets $R_n$ and $L_{n+1}$ with the least cardinality is maximal possible satisfying condition (c), and the second one is minimal such that $\card(D_n)\leq h(\max L_{n+1})$ (if $\card(R_n)<\card(L_{n+1})$) or $h(\min R_n)\leq h(\max L_{n+1})$ (if $\card(R_n)>\card(L_{n+1})$).
\end{itemize}
\end{itemize}
Observe that this construction is possible. Indeed, we have $R_n\cap L_n=\emptyset$ (since $\mu_n([\min(D_n),\min(D_n)+\lceil\card(D_n)/2\rceil])\geq 1/2$). Moreover, $R_n\subseteq L^n_1\cup\bigcup_{k>0}R^n_k$ (since $1\geq\mu_n(\bigcup_{k>1}L^n_k)\geq 2\card(\bigcup_{k>1}L^n_k)/\card(D_n)$) and $L_n\subseteq R^n_1\cup\bigcup_{k>0}L^n_k$ (since $\mu_n(\bigcup_{k>1}R^n_k)\leq \card(\bigcup_{k>1}R^n_k)/(2\card(D_n))\leq 1/2$). It remains to observe that $h(\min (L^n_1\cup\bigcup_{k>0}R^n_k))\leq\card(D_n)\leq\card(D_{n+1})/2\leq h(\max R^n_1\cup\bigcup_{k>0}L^n_k)$.

Let $r_n\in\omega$ be such that $(r_n+1)\card(D_n)\leq h(\max L_{n+1})$ and $(r_n+2)\card(D_n)>h(\max L_{n+1})$. Denote also $L=\bigcup_{n\in\omega}L_n$ and $R=\bigcup_{n\in\omega}R_n$. Define $g\colon\omega\to[0,\infty)$ by 
$$g(i)=\left\{\begin{array}{ll}
h(i) & \textrm{if }i\in D_n\setminus (L_n\cup R_n)\textrm{ for some }n\in\omega,\\
h(\max L_{n}) & \textrm{if }i\in L_n\textrm{ for some }n\in\omega,\\
(r_n+1)\card(D_n) & \textrm{if }i\in R_n\textrm{ for some }n\in\omega.
\end{array}\right.$$
Obviously, $g$ is nondecreasing and $\lim_{i\to\infty}g(i)=\infty$. We will show that $\I=\Z_g$. It will imply that $i/g(i)$ does not converge to $0$ as the latter is equivalent to $\omega\notin\Z_g$ (cf. the definition of simple density ideals) and we assumed that $\omega\notin\I$.

($\subseteq$): Suppose first that for some $B\subseteq\omega$ we have $B\notin\Z_g$. By Lemma \ref{lem3} we can assume that there are $\delta>0$ and a sequence $(l_j)$ such that $\frac{\card(B\cap D_{n_j}\cap l_j)}{g(l_j)}\geq\delta$, where $l_j\in D_{n_j}$. 

Consider first the case $B\subseteq\omega\setminus R$. Fix $j\in\omega$. By the definition of $g$, we have $\mu_{n_j}(B)\geq\frac{\card(B\cap D_{n_j}\cap l_j)}{g(l_j)}\geq\delta$. Hence, $B\notin\I$.

Suppose now that $B\subseteq R$. Then we can assume that $l_j\in R_{n_j}$. Let $C_j\subseteq L_{n_j+1}$ be any set of cardinality $\min(\card(B\cap R_{n_j}\cap l_j),\card(L_{n_j+1}))$ and define $C=\bigcup_{j\in\omega}C_j$. 

If $\card(C_j)<\card(B\cap R_{n_j}\cap l_j)$, then $\mu_{n_j+1}(C)>1/2$. On the other hand, if $\card(C_j)=\card(B\cap R_{n_j}\cap l_j)$, then we have $\frac{\card(B\cap R_{n_j}\cap l_j)}{g(l_j)}=\frac{\card(C_j)}{(r_{n_j}+1)\card(D_{n_j})}>\delta$ and $(r_{n_j}+2)\card(D_{n_j})>h(\max L_{n_j+1})$. Hence, 
$$\mu_{n_j+1}(C)\geq\frac{\card(C_j)}{h(\max L_{n_j+1})}>\frac{\delta (r_{n_j}+1)\card(D_{n_j})}{h(\max L_{n_j+1})}>\frac{\delta}{2}.$$ 
Therefore, $B\notin\I$ since $\I$ is increasing-invariant.

In the general case (i.e., if $B$ intersects both $R$ and $\omega\setminus R$), by the above considerations, either $B\cap R\notin\I$ (if $B\cap R\notin\Z_g$) or $B\cap (\omega\setminus R)\notin\I$ (if $B\cap (\omega\setminus R)\notin\Z_g$). In both cases we get that $B\notin\I$.

($\supseteq$): Suppose now that $B\in\Z_g$. First we deal with the case $B\subseteq\omega\setminus L$. Assume that $\card(B\cap D_n\cap i)/g(i)<1/m$ for some $n,m\in\omega$ and all $i\in D_n$. 

Observe first that at most $\card(D_n)/m$ elements of $\bigcup_{k\geq 1}L^n_k$ are in $B$. Simultaneously, at most $\card(D_n)/(2m)$ elements of $\bigcup_{k\geq 2}L^n_k$ are in $B$. Hence, $\mu_n(B)$ will be the biggest possible if $\card(B\cap L^n_1)\leq\frac{\card(D_n)}{2m}$. Similarly, we can show that $\mu_n(B)$ will be the biggest possible if for each $k\geq 1$ we will have $\card(B\cap L^n_k)\leq\frac{\card(D_n)}{m}(\frac{1}{k}-\frac{1}{k+1})=\frac{\card(D_n)}{mk(k+1)}$. Therefore, $B\cap \bigcup_{k\geq 1}L^n_k$ belongs to $\I$ since $\I$ is almost uniformly distributed.

We will show that $\mu_n(B\cap \bigcup_{k\geq 1}R^n_k)\leq\sum_{k=1}^m \frac{1}{mk}$. It will end this case as $\lim_{m\to\infty}\sum_{i=1}^m \frac{1}{mi}=0$. Note that $\mu_n(B)$ will be the biggest possible if $\card(B\cap R^n_k)\leq\frac{\card(D_n)}{m}$ for all $k\geq 1$ (by the same reasons as above). Therefore, we have $\mu_n(B\cap R^n_k)\leq\frac{\card(D_n)}{m}\frac{1}{k\card(D_n)}=\frac{1}{mk}$. Hence, $\mu_n(B\cap \bigcup_{k\geq 1}R^n_k)\leq\sum_{k=1}^m \frac{1}{mk}$.

Assume now that $B\subseteq L$. Then $\lim_{n\to\infty}\frac{\card(B\cap D_{n}\cap \max L_{n})}{h(\max L_{n})}=0$ (we restrict our considerations only to $n$'s such that $L_n\neq\emptyset$). Let $C_{n-1}\subseteq R_{n-1}$ be any set of cardinality $\min(\card(B\cap D_{n}\cap i_n),\card(R_{n-1}))$. 

Recall that $(r_{n-1}+2)\card(D_{n-1})>h(\max L_n)$ for all $n>1$. We have 
$$\mu_{n-1}(C_{n-1})\leq\frac{\card(C_{n-1})}{(r_{n-1}+1)\card(D_{n-1})}<2\frac{\card(B\cap D_{n}\cap i_n)}{h(\max L_n)}.$$ 
Hence, $\lim_{n\to\infty}\mu_{n}(C)=0$ and $C\in\I$. 

Note that $\card(L_n)\leq\card(D_{n-1})+1$ for each $n$. Indeed, it follows from $1\geq\mu_n(L_n\setminus\max L_n)\geq(\card(L_n)-1)\mu_n(\{\max L_n-1\})\geq (\card(L_n)-1)/h(\max L_n-1)$ and $h(\max L_n-1)<\card(D_{n-1})\leq h(\max L_n)$. Recall also that in the case of $\card(C_{n-1})<\card(B\cap D_{n}\cap i_n)$ we have $\card(R_{n-1})=\lfloor\card(D_{n-1})/2\rfloor$. Therefore, $\card(B\cap D_{n}\cap i_n)\leq 2\card(C_{n-1})+3$ for all $n$. It suffices to observe that increasing-invariance of $\I$ implies $B\in\I$ (cf. Proposition \ref{delta}). This ends the entire proof.
\end{proof}

\section{Open problems}

We end our paper with four open problems. 

Firstly, recall that every Erd\H{o}s-Ulam ideal is $\leq_{RB}$-equivalent to $\Z$ (by \cite[Theorem 1.13.10]{Farah}) and a simple density ideal is an Erd\H{o}s-Ulam ideal if and only if it is $\leq_{RB}$-equivalent to $\Z$ (by our Corollary \ref{cor}). Moreover, by Remark \ref{EU:RB<->K} in the latter result we can replace $\leq_{RB}$ by $\leq_K$. Therefore, it is natural to ask the following.

\begin{problem}
Does $\leq_{RB}$-equivalence of $\I$ and $\Z$ imply that $\I$ is an Erd\H{o}s-Ulam ideal? If yes, can $\leq_{RB}$ be replaced by $\leq_K$?
\end{problem}

Secondly, we wonder which Erd\H{o}s-Ulam ideals are isomorphic to some simple density ideal. It is easy to see that almost uniform distribution is a necessary condition since it is preserved by isomorphisms. However, the next example shows that increasing-invariance is not preserved by isomorphisms (even in case of Erd\H{o}s-Ulam ideals). 

\begin{ex}
\label{rsi-iso}
Let $(D_n)$ be a sequence of consecutive intervals such that $\card(D_n)=2n!$ for all $n\in\omega$. Define
$$\mu_n(\{i\})=\left\{\begin{array}{ll}
1/n! & \textrm{if } \min D_n\leq i<\min D_n+n!,\\
0 & \textrm{otherwise;}
\end{array}\right.$$
$$\nu_n(\{i\})=\left\{\begin{array}{ll}
1/n! & \textrm{if } \max D_n-n!<i\leq\max D_n,\\
0 & \textrm{otherwise.}
\end{array}\right.$$
Then $\Exh(\sup_{n\in\omega}\mu_n)$ is increasing-invariant, $\Exh(\sup_{n\in\omega}\nu_n)$ is not increasing-invariant, both are Erd\H{o}s-Ulam ideals and they are isomorphic. 
\end{ex}

In this context it is natural to ask whether every Erd\H{o}s-Ulam ideal has an increasing-invariant isomorphic copy. The next example shows that this is not true even if we assume almost uniform distribution.

\begin{ex}
Let $(D_n)$ be a sequence of consecutive intervals such that $\card(D_n)=k$ for all $n\in(k(k-1)/2,k(k+1)/2]$ (we put $D_0=\emptyset$). Define $\mu_n(\{i\})=1/\card(D_n)$ for all $i\in D_n$ and $\mu_n(\{i\})=0$ otherwise. Fix any bijection $\phi\colon\omega\to\omega$ and denote $\I=\Exh(\sup_{n\in\omega}\mu_n)$ and $\J=\{\phi[A]:\ A\in\I\}$. Notice that $\I$ is almost uniformly distributed by Proposition \ref{bounded}.

Let $b_n=\min \phi[D_n]$ for all $n\in\omega$. Then $B=\{b_n:\ n\in\omega\}$ belongs to $\J$. Define also $l_k\in(k(k-1)/2,k(k+1)/2]$ by $b_{l_k}\geq b_n$ for all $n\in(k(k-1)/2,k(k+1)/2]$. Consider the set $C=\bigcup_{k\in\omega}\phi[D_{l_k}]\notin\J$. There is a bijection $\sigma\colon C\to B$ such that $\sigma(i)\leq i$ for all $i\in C$ (it suffices to assure that if $i\in\phi[D_{l_k}]$ then $\phi(i)=b_n$ for some $n\in(k(k-1)/2,k(k+1)/2]$). Then we have $\card(C\cap n)=\card(\sigma[C\cap n]\cap n)\leq\card(B\cap n)$. Hence, $\J$ is not increasing-invariant.
\end{ex}

Therefore, the question about Erd\H{o}s-Ulam ideals which are isomorphic to some simple density ideal actually reduces to the following one.

\begin{problem}
Which Erd\H{o}s-Ulam ideals are isomorphic to some increasing-invariant ideals?
\end{problem}

We also want to ask about characterization of density ideals which are simple density ideals. One small step towards answering this question has recently been made in \cite[Proposition 10]{generalized}, where it is shown that every density ideal satisfying some strong technical assumptions (for instance, all measures generating it are uniformly distributed) is a simple density ideal. We do not know if our characterization from the previous Theorem \ref{main} works in the case of all tall density ideals (it is easy to see that it does not work for non-tall ideals: the ideal $\Fin$ is increasing-invariant and almost uniformly distributed, however it cannot be a simple density ideal as it is $\mathbf{F}_{\sigma}$ -- cf. Remark \ref{rem}).

\begin{problem}
Is every tall, increasing-invariant and almost uniformly distributed density ideal a simple density ideal?
\end{problem}

The last open problem concerns the following notion proposed recently in \cite{BoseDas}. A function $f\colon\mathbb{R}_+\cup\{0\}\to\mathbb{R}_+\cup\{0\}$ is called \emph{modulus function} (cf. \cite{1}) provided that:
\begin{itemize}
	\item $f(x)=0$ if and only if $x=0$;
	\item $f(x+y)\leq f(x)+f(y)$ for all $x,y\in\mathbb{R}_+\cup\{0\}$;
	\item $f$ is increasing and right-continuous at $0$.
\end{itemize}
Now, for each $g\in H$ and an unbounded modulus function $f$ the family $\Z_g(f)=\{A\subseteq\omega:\ \lim_{n\to\infty}\frac{f(\card(A\cap n))}{f(g(n))}=0\}$ is an ideal. By \cite[Theorem 3.7]{BoseDas}, every ideal of the form $\Z_g(f)$ is a density ideal. Obviously, every simple density ideal is of the form $\Z_g(f)$ as witnessed by the identity modulus function. What is more, $\Z_{\id}(\log(x+1))$ is an example of an ideal of the form $\Z_g(f)$ which is not a simple density ideal (cf. \cite[Example 2.3]{BoseDas}). 

Ideals of the form $\Z_{g}(f)$ were introduced after obtaining main results of this paper. Thus, our considerations do not include them. The ideal $\Z_{\id}(\log(x+1))$ does not seem to be an Erd\H{o}s-Ulam ideal. Therefore, we propose the following open problem. A negative answer would mean that our characterization from Theorem \ref{main} works also in the case of ideals of the form $\Z_{g}(f)$.

\begin{problem}
Is there an Erd\H{o}s-Ulam ideal of the form $\Z_{g}(f)$ that is not a simple density ideal?
\end{problem}

\end{document}